\documentclass[a4paper,11pt]{amsart}

\usepackage{epsfig}
\usepackage{amsthm,amsfonts}
\usepackage{amssymb,graphicx,color}
\usepackage[all]{xy}
\usepackage{verbatim}
\usepackage{hyperref}
\usepackage{enumitem}

\newtheorem{theorem}{Theorem}[section]
\newtheorem*{theorem*}{Theorem}

\newtheorem{definition}[theorem]{Definition}
\newtheorem{conjecture}{Conjecture}

\newtheorem{remark}[theorem]{Remark}

\newcommand{\C}{\mathbb{C}}

\begin{document}

\title[On metric equivalence of the Brieskorn-Pham hypersurfaces]
{On metric equivalence of the Brieskorn-Pham hypersurfaces}

\author[A. Fernandes]{Alexandre Fernandes}
\author[Z. Jelonek]{Zbigniew Jelonek}
\author[J. E. Sampaio]{Jos\'e Edson Sampaio}

\address[Zbigniew Jelonek]{ Instytut Matematyczny, Polska Akademia Nauk, \'Sniadeckich 8, 00-656 Warszawa, Poland. \newline  
              E-mail: {\tt najelone@cyf-kr.edu.pl}
}

{\address[Jos\'e Edson Sampaio \& Alexandre Fernandes]{    
              Departamento de Matem\'atica, Universidade Federal do Cear\'a,
	      Rua Campus do Pici, s/n, Bloco 914, Pici, 60440-900, 
	     Fortaleza-CE, Brazil. \newline  
               E-mail: {\tt edsonsampaio@mat.ufc.br}
}

\thanks{The first named author was partially supported by CNPq-Brazil grant 304700/2021-5. The second named author is partially supported by the grant of Narodowe Centrum Nauki number. The third named author was partially supported by CNPq-Brazil grant 310438/2021-7 and by the Serrapilheira Institute (grant number Serra -- R-2110-39576).
}
\begin{abstract}
We show that two bi-Lipschitz equivalent Brieskorn-Pham hypersurfaces have the same multiplicities at $0$. Moreover we show that 
if two algebraic $(n-1)$-dimensional cones $P, R\subset\mathbb C^n$ with isolated singularities are homeomorphic, then they have the same degree. 
\end{abstract}

\maketitle

\section{Introduction}

Recently, in \cite{BobadillaFS:2018} the following conjecture was proposed:
\begin{conjecture} \label{conj_local}
Let $X\subset \C^n$ and $Y\subset \C^m$ be two complex analytic sets with $\dim X=\dim Y=d$. If their germs at zero are bi-Lipschitz homeomorphic, then their multiplicities $m(X,0)$ and $m(Y,0)$ are equal.
\end{conjecture}
Already in \cite{BobadillaFS:2018} the authors proved that Conjecture \ref{conj_local} has a positive answer for $d=2$.
The positive answer for $d=1$ was already known, since it follows from the bi-Lipschitz classification of germs of complex analytic curves (see \cite{N-P}, \cite{P-T} and \cite{F}). 
However, in dimension  three, Birbrair at al.  \cite{bfsv} have presented examples of complex algebraic cones $X$ and $Y$ with isolated singularity, which were bi-Lipschitz homeomorphic but with different multiplicities at the origin. Recently, the authors of this paper in \cite{FernandesJS:2023} proved that for every $k\ge 3$ there exist complex algebraic cones of dimension $k$ with isolated singularities, which are bi-Lipschitz and semi-algebraically equivalent but have different multiplicities. 
See \cite{FernandesS:2023} for a survey on this conjecture.

In this paper, we prove Conjecture \ref{conj_local} in the special case of Brieskorn-Pham hypersurfaces, i.e., hypersurfaces described by the formula $V=\{ z\in \C^n: z_1^{a_1}+z_2^{a_2} +...+z_{n}^{a_{n}}=0\},$ where $a_i\ge 2.$
More specifically, we show that two bi-Lipschitz equivalent Brieskorn-Pham hypersurfaces have the same multiplicity at $0$ (see Theorem \ref{4}). Moreover, we show that if two algebraic $n-1$ dimensional cones $P, R\subset\mathbb C^n$ with isolated singularity are homeomorphic, then they have the same degree (see Theorem \ref{hyp}). 

Note that for $V=\{ (x,y)\in \C^2: x^2+y^3=0\}$ and $W=\{ (x,y)\in \C^2: x^3+y^4=0\}$, we have that $(V,0)$ and $(W,0)$ are homeomorphic, but $m(V,0)=2\not=3=m(W,0).$ Hence in Theorem \ref{4} we cannot avoid assumption that homeomorphism is bi-Lipschitz.

\vspace{5mm}

\section{Cones with smooth bases}
We start with a definition:

\begin{definition}
Let $X\subset \mathbb {CP}^n$ be an algebraic variety. We assume $\mathbb {CP}^n$  to be a hyperplane at infinity of $\mathbb {CP}^{n+1}$.  Then by an algebraic cone $\overline{C(X)}\subset \mathbb {CP}^{n+1}$ with base $X$ we mean the set
$$\overline{C(X)}=\bigcup_{x\in X} \overline{O,x},$$ where $O$ is the center of coordinates in $\mathbb C^{n+1}\subset \mathbb {CP}^{n+1}$, and $\overline{O,x}$ means the projective line which goes through $O$ and $x.$ By an affine cone $C(X)$ we mean $\overline{C(X)}\setminus X.$
By the link of $C(X)$ we mean the set $L=\{ x\in C(X): ||x||=1\}.$
\end{definition}

The following theorem is well-known (see e.g., \cite{max})

\begin{theorem}\label{hyp}
Let $V\subset \C\mathbb P^{n+1}$ be a smooth algebraic hypersurface. Then the integral (co)homology of V is torsion free, and the corresponding Betti numbers are given as follows:

(1) $b_i(V)=0$ for $i\not=n  \ odd \ or \ i\not\in [0,2n].$

(2) $b_i(V)=1$ for $i\not=n$\ even \ and $i\in [0,2n]. $

(3) $b_n(V)=\frac{(d-1)^{n+2}+(-1)^{n+1}}{d} + \frac{3(-1)^{n}+1}{2}.$
\end{theorem}

Now we are ready to prove:

\begin{theorem}\label{1}
Let $P=C(X),R=C(Y) \subset \C^{n+2}$ be two algebraic $n+1$ dimensional cones with smooth bases $X,Y.$ If $P, R$ are homeomorphic, then  ${\rm deg} \ P = {\rm deg} \ R.$ In particular $X$ and $Y$ are also homeomorphic.
\end{theorem}

\begin{proof}
Since theorem is true for $1-$dimensional cones, we can assume that $P,R$ have connected bases. By \cite{prill} we can assume that $0$ is not topologically regular point of $P$ and $R$, because otherwise both
cones are hyperplanes. Hence $P^*=P\setminus \{0\}$ is homeomorphic to $R^*=R\setminus\{0\}.$ Let $L_P, L_R$ be links of $P,R$. Hence $L_P$ is a deformation retract of $P^*$, similarly $L_R$ is a deformation retract of $R^*$, 
Since $P^*$ and $R^*$ are homeomorphic we have that $L_P$ is homotopically equivalent to $L_R.$ Let $X,Y$ be bases of $P,R.$

We have the Hopf fibration $\pi : L_P \to X,\pi': L_R\to Y$ whose fibers are circles.
Analizing  the spectral sequences of the mappings $\pi$ and $\pi'$ (see  \cite{orlik}, \cite{kol}), we have that the corresponding Betti numbers are given as follows:

1) $b_i(L_P, \mathbb Q) = b_i(X,\mathbb Q)-b_{i-2}(X,\mathbb Q)$  if $i\le {\rm dim} \ X$ 

2)  $b_{i+1}(L_P, \mathbb Q) = b_i(X,\mathbb Q)-b_{i+2}(X,\mathbb Q)$ if $i\ge {\rm dim} X.$

3) $b_i(L_R, \mathbb Q) = b_i(Y,\mathbb Q)-b_{i-2}(Y,\mathbb Q)$  if $i\le {\rm dim} \ X$ 

4)  $b_{i+1}(L_R, \mathbb Q) = b_i(Y,\mathbb Q)-b_{i+2}(Y,\mathbb Q)$ if $i\ge {\rm dim} X.$

Let deg $P=p$, deg $R=r.$ Since $P,R$  have homotopic links $L_P\sim L_R$, we see by 1) - 4)  and Theorem \ref{hyp}  that
$b_n(X,\mathbb Q)=b_n(Y,\mathbb Q)$, i.e., $$\frac{(p-1)^{n+2}+(-1)^{n+1}}{p}=\frac{(r-1)^{n+2}+(-1)^{n+1}}{r}.$$ Since the function $f(x)=\frac{(x-1)^{n+2}+(-1)^{n+1}}{x}$ increases for $x\ge 1$ we have $p=r.$
\end{proof}

In the same way we can prove:

\begin{theorem}\label{2}
Let $P,R $ be two algebraic  cones with smooth bases $X,Y$. If $P,R$  are homeomorphic, then  $X$ and $Y$ have the same Betti numbers. In particular $\chi(X)=\chi(Y).$
\end{theorem}

\begin{proof}
As above, by \cite{prill} the spaces  $P^*$ and $R^*$ are also  homeomorphic. In particular, the links $L_P$, $L_R$ are homotopic.  By formulas 1)-4)  above we see that $X$, $Y$ have the same Betti numbers.
\end{proof}

\begin{remark}
{\rm If $S\subset \mathbb{C}^m$ is a homogeneous complex algebraic surface, it was proved in \cite{BobadillaFS:2018} that the torsion of $H^2(S\setminus \{0\})$ is equal to $\mathbb{Z}/{\rm deg} S \mathbb{Z}$. However, this cannot be extended to higher dimension. Indeed, let $Q=\mathbb {CP}^{1}\times\mathbb {CP}^{1}\subset \mathbb {CP}^{3}$ be the quadric and $L$ the link of $C(Q)$. Then $L$ is simply connected (see the proof of Theorem 3.1 in \cite{FernandesJS:2023}). By \cite[Proposition 2.4]{bfsv}, $L$ is diffeomorphic to $\mathbb{S}^2\times\mathbb{S}^3$.  Hence by K\"unneth Formula, $L$ has free homology (and cohomology). Therefore the cohomology of $C(Q)\setminus\{0\}$ has no torsion.}
\end{remark}

\section{Equivalence of the Brieskorn-Pham hypersurfaces}

Here we prove 

\begin{theorem}\label{4}
Let $V=\{ z \in \C^{n}: z_1^{a_1}+...+z_n^{a_n}=0\}, W=\{ z\in \C^n: z_1^{b_1}+....+z_n^{b_n}=0\},$  where $2\le a_1=a_2=...=a_k<....\le a_n$ and $2\le b_1= b_2=...=b_l<....\le b_n$. Assume that $(V,0)$ is bi-Lipschitz equivalent to $(W,0).$ Then

1) $l=k$

2) $a_1=b_1.$

\noindent In particular $m(V,0)=m(W,0).$
\end{theorem}

\begin{proof}
If $V$ is bi-Lipschitz equivalent to $W$, then by \cite{Sampaio:2016} also the tangent cones $T(V)=\{z\in \C^n: z_1^a+...+z_k^a=0\}$, $T(W)=\{z\in \C^n: z_1^b+...+z_l^b=0\}$ (where $a=a_1$ and $b=b_1$) are bi-Lipschitz equivalent.
But Sing $T(V)=\{0\}\times \C^{n-k}$, Sing $T(W)=\{0\}\times \C^{n-l}$. By \cite[Theorem 4.2]{Sampaio:2016} (see also \cite{bir}), we have $n-k=n-l$, hence $k=l.$ The case $l\le 2$ follows from   \cite[Proposition 1.6]{FernandesS:2016} (see also \cite[Theorem 5.3]{FernandesS:2023}). Hence we can assume that tangent cones are irreducible and reduced. Moreover $T(V)\setminus {\rm Sing} \ T(V)=P^*\times\C^{n-k}$ is homeomorphic to $T(W)\setminus {\rm Sing}\  T(W)=R^*\times\C^{n-k},$ 
where $P=\{ w\in \C^{k}: w_1^a+...+w_k^a=0\}$, $R=\{ w\in \C^{k}: w_1^b+...+w_k^b=0\}$.
This means that $P^*$ and $R^*$ are homotopic. Hence also $L_P$ and $L_R$ are homotopic. Now we end as in the proof of Theorem \ref{1} to obtain  $a=b.$
\end{proof}



\begin{thebibliography}{99}
\bibitem{bir}
{Birbrair, L.; Fernandes, A.; L\^e D. T. and Sampaio, J. E.}
{\em Lipschitz regular complex algebraic sets are smooth}.
Proc. Amer. Math. Soc., vol. 144 (2016), no. 3, 983-987.

\bibitem{bfsv} Birbrair, L.; Fernandes, A.; Sampaio, J. E. and  Verbitsky, M. 
{\it Multiplicity of singularities is not a bi-Lipschitz invariant}. 
Math. Ann., vol. 377 (2020), 115-121.

\bibitem{BobadillaFS:2018} 
Bobadilla, J.F. de; Fernandes, A. and Sampaio, J. E. 
{\it Multiplicity and degree as bi-lipschitz invariants for complex sets}.
Journal of Topology, vol. 11 (2018), 958--966.

\bibitem{F}
{Fernandes, A.}
{\it Topological equivalence of complex curves and bi-Lipschitz maps}.
{Michigan Math. J.}, vol. 51 (2003), 593--606.

\bibitem{FernandesS:2016}
{Fernandes, A. and Sampaio, J. E.}.
{\it Multiplicity of analytic hypersurface singularities under bi-Lipschitz homeomorphisms}.
Journal of Topology, vol. 9 (2016), 927--933.

\bibitem{FernandesS:2023} Fernandes, A. and Sampaio, J.E. , {\it Bi-Lipschitz Invariance of the Multiplicity.} In: Cisneros-Molina, J.L., D\~ung Tr\'ang, L., Seade, J. (eds) Handbook of Geometry and Topology of Singularities IV. Springer, 2023.

\bibitem{FernandesJS:2023}
Fernandes, A.; Jelonek, Z. and Sampaio, J. E.
{\it Bi-Lipschitz equivalent cones with different degrees}.
Preprint (2023), arXiv:2309.07078 [math.AG].


\bibitem{kol} Kollar, J., 
{\it Links of complex analytic singularities}.  In: Surveys in Differential Geometry XVIII.
International Press of Boston, (2013), 157--192.
 
\bibitem{max} Maxim, L.G., 
{\it On the topology of complex projective hypersurfaces}. 
Res. Math. Sci., vol. 11 (2024), article no. 14, 1--22.

\bibitem{N-P}
{Neumann, W. and Pichon, A.}
{\it Lipschitz geometry of complex curves}.
 Journal of Singularities, vol. 10 (2014), 225--234.

\bibitem{orlik} Orlik, P. and  Wagreich, P.
{\it Seifert n-manifolds}, Invent. Math., vol. 28 (1975),
137--159.

\bibitem{P-T}
{Pham, F. and Teissier, B.}
{\it Fractions lipschitziennes d'une alg\`ebre analytique complexe et saturation de Zariski}.
{Pr\'epublications du Centre de Math\'ematiques de l'Ecole Polytechnique (Paris)}, no. M17.0669,
June (1969). Available at \url{https://hal.archives-ouvertes.fr/hal-00384928/}

\bibitem{prill} Prill, D.
{\it Cones in Complex Affine Space are Topologically Singular}. 
Proc. Amer. Math. Soc., vol. 18, No. 1 (1967), pp.
178--182.

\bibitem{Sampaio:2016}
{Sampaio, J. E.}
{\it Bi-Lipschitz homeomorphic subanalytic sets have bi-Lipschitz homeomorphic tangent cones}.
Selecta Math. (N.S.), vol. 22 (2016), no. 2, 553--559.
\end{thebibliography}
\end{document}